\documentclass[10pt]{article}
\DeclareMathSymbol{\varChi}{\mathord}{letters}{88}
\usepackage[latin1]{inputenc}
\usepackage{amsfonts}
\usepackage{mathrsfs}
\usepackage{amsmath}
\usepackage{amssymb}
\usepackage{amsthm}
\usepackage{tensor}
 \usepackage{color}
\definecolor{red}{rgb}{1.0,0.0,0.0}

\definecolor{blu}{rgb}{0.0,0.0,1.0}
\def\blu#1{{\textcolor{blu}{#1}}}
\definecolor{gre}{rgb}{0.03,0.50,0.03}

\usepackage[T1]{fontenc}

\newcommand{\ud}{\,\mathrm{d}}

\def\M{\mathbb M}
\def\R{\mathbb R}
\def\X{\mathbb X}
\def\Y{\mathbb Y}
\def\N{\mathbb N}
\def\W{\mathbb W}
\def\Z{\mathbb Z}
\def\A{\mathbb A}

\def\V{\mathbb V}
\def\Q{\mathbb Q}
\def\P{\mathbb P}

\newtheorem{Theorem}{Theorem}[section]
\newtheorem{Definition}[Theorem]{Definition}
\newtheorem{Proposition}[Theorem]{Proposition}
\newtheorem{Lemma}[Theorem]{Lemma}

\newtheorem{Hypothesis}[Theorem]{Hypothesis}
\newtheorem{Remark}[Theorem]{Remark}
\newtheorem{Example}[Theorem]{Example}

\selectfont

\author{Giorgio Fabbri\footnote{Aix-Marseille Univ. (Aix-Marseille School of Economics), CNRS \& EHESS. 5, Boulevard Maurice Bourdet, 13205 Marseille Cedex 01, France. E-mail: giorgio.fabbri@univ-amu.fr. 
The work of this author has been developed in the framework of the center of excellence LABEX MME-DII (ANR-11-LABX-0023-01).}
 \; and \;  Francesco Russo\footnote{ENSTA ParisTech, Universit\'e Paris-Saclay,  Unit\'e de Math\'ematiques appliqu\'ees, 828, Boulevard des Mar\'echaux,
 F-91120 Palaiseau, France. E-mail: francesco.russo@ensta-paristech.fr. 
 The financial support of this author was partially provided
 by the DFG 
through the CRC ''Taming uncertainty and profiting from randomness and
low regularity in analysis, stochastics and their application''.}}

\title{HJB equations in infinite dimension and optimal
control of stochastic evolution equations
 via generalized Fukushima decomposition }
\date{August 1st 2017}

\begin{document}

\maketitle

\begin{abstract}
A stochastic optimal control problem driven by an abstract evolution equation in a separable Hilbert space is considered. Thanks to the identification of 
the mild solution of the state equation as $\nu$-weak Dirichlet process,
  the value processes is proved to be a real weak Dirichlet process.
 The uniqueness of the corresponding decomposition  is used to prove a
 verification theorem. 

Through that technique several of the required assumptions are milder than 
those employed in previous contributions about 
non-regular solutions of Hamilton-Jacobi-Bellman equations.
\end{abstract}

\bigskip

\bigskip

{\bf KEY WORDS AND PHRASES:} Weak Dirichlet processes in infinite dimension;  
Stochastic evolution equations; Generalized Fukushima decomposition; Stochastic optimal control in Hilbert spaces.

\medskip

{\bf 2010 AMS MATH CLASSIFICATION:} 35Q93, 93E20, 49J20


\newpage

\section{Introduction}

The goal of this paper is to show that, if we carefully exploit some
 recent developments in stochastic calculus in infinite dimension, 
we can weaken some of the hypotheses typically demanded in the literature of
 non-regular solutions of Hamilton-Jacobi-Bellman (HJB) equations to prove
 verification theorems and optimal syntheses of stochastic optimal control problems in Hilbert spaces.

As well-known,  the study of a dynamic optimization problem can be linked, via the dynamic programming to the analysis of the related HJB equation, that is, in the context we are interested in, a second order infinite dimension PDE. When this approach can be successfully applied, one can prove a verification theorem and express the optimal control in feedback form (that is, at any time, as a function of the state) using the  solution of the HJB equation. In this case the latter can be identified with the value function of the problem.

In the regular case (i.e. when the value function is $C^{1,2}$, see for instance Chapter 2 of \cite{FabbriGozziSwiech16}) the standard proof of the verification theorem is based on the It\^o formula.
In this paper we show that some recent results in stochastic calculus, in particular Fukushima-type decompositions explicitly suited for the infinite dimensional context, can be used to prove the same kind of result for less regular solutions of the HJB equation.

\smallskip

The idea is the following.
In a previous paper (\cite{FabbriRusso16}) the authors introduced the class of \emph{$\nu$-weak Dirichlet processes} (the definition is recalled in Section \ref{sec:preliminaries}, $\nu$ is a Banach space
strictly associated with a suitable subspace $\nu_0$ of $H$) 
and showed that convolution type processes, and in particular mild solutions
 of infinite dimensional stochastic evolution
 equations (see e.g. \cite{DaPratoZabczyk14}, Chapter 4), belong to this class. By applying this result to the solution of the state equation of a class of stochastic optimal control problems in infinite dimension we are able to show that the value process, that is the value of any given solution of the HJB equation computed on the trajectory taken into account\footnote{
The expression \emph{value process} is sometime used for denoting the value function computed on the trajectory, often the two definition coincide but it is not always the case.}, is a  (real-valued) {\it weak Dirichlet processes} (with respect to a given filtration), a notion introduced in \cite{errami} and subsequently analyzed in  \cite{GozziRusso06}.  Such a process can be written as the sum of a local martingale and a  martingale
orthogonal process, i.e.  having zero covariation with every continuous local martingale. Such decomposition is 
unique and in Theorem \ref{th:decompo-sol-HJB},
we exploit the uniqueness property to characterize  the martingale part of
 the value process as a suitable stochastic integral
 with respect to a Girsanov-transformed Wiener process
which allows  to obtain a substitute of the It\^o-Dynkin formula for solutions of the Hamilton-Jacobi-Bellman equation.
This is possible when the value process associated to the optimal control problem   can be expressed by  
a $C^{0,1}([0,T[ \times H)$ function of the state process, with however a stronger regularity on the first derivative.
 We finally use this expression to prove  the verification result stated in Theorem \ref{th:verification}\footnote{A similar approach is used, when $H$ is  finite-dimensional,
 in \cite{GozziRusso06Stoch}. In that case 
things are simpler and there is not need to use the notion of $\nu$-weak Dirichlet processes and and results that are specifically suited for the infinite dimensional case. In that case $\nu_0$ will be isomorphic to the full space $H$.}.

\medskip

We think the interest of our contribution is twofold. On the one hand we show that recent developments in stochastic 
calculus in Banach spaces, see  for instance \cite{DiGirolamiRusso11Fukushima, DiGirolamiRusso11}, from which we adopt the 
framework related to  generalized covariations  and It\^o-Fukushima formulae, but also other approaches as \cite{cosso_russo15a, leao_ohashi_simas14, vanNeervenVeraarWeis07} 
may have important control theory counterpart applications.
On the other hand the method we present allows to improve some previous verification results weakening a series of hypotheses. 

We discuss here this second point in detail.
There are several ways to introduce non-regular solutions of second order HJB equations in Hilbert spaces. They are more precisely surveyed in \cite{FabbriGozziSwiech16} but they essentially are viscosity solutions, strong solutions and the study of the HJB equation through backward SDEs.
\emph{Viscosity solutions} are defined, as in the finite-dimensional case, using test functions that 
locally ``touch'' the candidate solution. The viscosity solution approach was first adapted to the second order Hamilton 
Jacobi equation in Hilbert space in \cite{Lions83I,Lions83II,Lions83III} and then, for the ``unbounded'' case (i.e. including a possibly unbounded generator of a strongly continuous semigroup in the state equation, see e.g. equation (\ref{eq:state})) in \cite{Swiech94}. 
Several improvements of those pioneering studies have been published, including extensions to several specific 
equations but, differently from  what happens in the finite-dimensional case, there are no verification theorems available at the moment for stochastic problems in infinite-dimension that use the notion of viscosity solution.
The \emph{backward SDE approach} can be applied when the mild solution of the HJB equation can be represented 
using the solution of a forward-backward system.
 It was introduced in \cite{Quenez97} 
in the finite dimensional setting
and developed in several works,
 among them \cite{DeFuTe, FuHuTe-quadr,FuhrmanMasieroTessitore11, FuhrmanTessitore02-ann, FuTe-ell}. This method 
only allows to find optimal feedbacks in classes of problems satisfying a specific ``structural condition'', imposing,
 roughly speaking, that the control acts within the image of the noise. The same limitation concerns the \emph{$L^2_{\mu}$ approach} introduced and developed in \cite{Ahmed03} and \cite{GoldysGozzi06}.

In the \emph{strong solutions} approach, first introduced in \cite{BarbuDaPrato83book}, the solution is defined as a 
proper limit of solutions of regularized problems.  Verification results in this framework are given in 
\cite{Gozzi95,Gozzi96,Gozzi02,GozziRouy96}. They are collected and refined in Chapter 4 of \cite{FabbriGozziSwiech16}. 
The results obtained  using \emph{strong solutions} are the main term of comparison for ours both because in this context the verification results are more developed and because we partially work in the same framework by approximating the solution of the HJB equation using solutions of regularized problems. With reference to them our method has some advantages \footnote{Results for specific cases, as boundary control problems and reaction-diffusion equation (see \cite{Cerrai01,Cerrai01-40}) cannot be treated at the moment with the method we present here.}: (i) the assumptions on the cost structure are milder, notably they do not include any continuity assumption on the running cost that is only asked to be a measurable function; moreover the admissible controls are only asked to verify, together with the related trajectories, a quasi-integrability condition of the functional, see Hypothesis \ref{hp:on-l-and-g} and the subsequent paragraph; (ii) we work with a bigger set of approximating functions because we do not require the approximating functions and their derivatives to be uniformly
bounded; (iii)  the convergence of the derivatives of the approximating solution is not necessary and it is replaced by the weaker condition  (\ref{eq:b}). 
This convergence, in different possible forms, is unavoidable in the standard structure of the strong solutions approach and it is avoided here only thanks to the use of Fukushima decomposition in the proof. 
In terms of the last just mentioned two points, our notion of solution is  weaker than those used in the mentioned works,
 we need nevertheless to assume that the gradient of the solution of the HJB equation is continuous as an $D(A^*)$-valued function.

 Even if it is rather simple, it is itself of some interest because, as far as we know, no explicit (i.e. with explicit expressions of the value function and of the approximating sequence) example of \emph{strong solution} for second order HJB in infinite dimension are published so far.

\medskip

The paper proceeds as follows. Section \ref{sec:preliminaries} is devoted to some preliminary notions, notably 
the definition of $\nu$-weak-Dirichlet process and some related results. Section \ref{sec:HJB} focuses on the optimal
 control problem and the related HJB equation. It includes  the key 
decomposition Theorem \ref{th:decompo-sol-HJB}. Section \ref{sec:verification}  concerns the verification theorem. In Section \ref{sec:example} we provide an example of optimal control problem that can solved by using the developed techniques.

\section{Some preliminary definitions and result}
\label{sec:preliminaries}
Consider a complete probability space $\left( \Omega,\mathscr{F},\mathbb{P}\right)$. Fix $T>0$ and $s\in [0,T[$. Let $\left \{ \mathscr{F}^s_t \right \}_{t\geq s}$ be a filtration satisfying the usual conditions.  Each time we use expressions as ``adapted'', ``martingale'', etc... we always mean ``with respect to the filtration $\left \{ \mathscr{F}^s_t \right \}_{t\geq s}$''.

Given a metric space $S$ we denote by $\mathscr{B}(S)$ the Borel $\sigma$-field on $S$. Consider two real Hilbert spaces $H$ and $G$. By default we assume that all the processes $\X \colon [s,T]\times \Omega \to H$ are
Bochner  measurable functions  with respect to the product $\sigma$-algebra
 $\mathscr{B}([s,T]) \otimes \mathscr{F}$ with values in $(H, \mathscr{B}(H))$.
Continuous processes are clearly Bochner measurable processes.
 Similar conventions are done  for $G$-valued processes. We denote by $H\hat\otimes_\pi G$ the projective tensor  product of $H$ and $G$, see \cite{Ryan02} for details.

\begin{Definition}
\label{def:Dirweak}
A continuous real process $X \colon [s,T]\times \Omega \to \mathbb{R}$ is called
 {\bf weak Dirichlet process} if it can be written as $X=M+A$, where $M$ is a continuous local martingale and $A$ is a {\bf martingale orthogonal process} 
 in the sense  that $A(s)=0$ and $\left[ A,N\right] =0$ for every continuous local martingale $N$.
\end{Definition}
The following result is proved in Remarks 3.5 and 3.2 of \cite{GozziRusso06}.
\begin{Theorem}
\label{th:uniqueness-weak-Dirichlet}
\begin{enumerate}
\item The decomposition described in Definition \ref{def:Dirweak} is unique.
\item A semimartingale is a weak Dirichlet process.
\end{enumerate}
\end{Theorem}
The notion of weak Dirichlet process constitutes a natural generalization 
of the one of semimartingale. To figure out this fact one can start by considering a real continuous semimartingale $S = M + V$,  where $M$ is a local martingale and $V$ is a bounded variation
process vanishing at zero. Given a function $f: [0,T] \times \R \rightarrow \R$
of class $C^{1,2}$, It\^o formula shows that 
\begin{equation} \label{EMF}
f(\cdot, S) = M^f + A^f
\end{equation}
is a semimartingale where $M^f_t = f(0,S_0) + \int_0^t \partial_x f(r,S_r) dM_r$
is a local martingale and $A^f$ is a bounded variation process expressed
in  terms of the partial derivatives of $f$.
If $f \in C^{0,1}$ then \eqref{EMF} still holds with the same $M^f$, but now $A^f$ is only a martingale orthogonal process; in this case $f(\cdot, S)$ is generally no longer a semimartingale
but only a weak Dirichlet process, see \cite{GozziRusso06}, Corollary 3.11.
 For this
reason \eqref{EMF} can be interpreted as a generalized It\^o formula.

Another aspect to be emphasized is that  a semimartingale is also a finite quadratic variation process.
Some authors, see e.g. \cite{MetivierPellaumail80, Dinculeanu00} have extended the notion
of quadratic variation to the case of stochastic process taking values in
a Hilbert (or even Banach) space $B$. 
The difficulty is that 
the notion of  finite quadratic variation process (but also the one 
of semimartingale or weak Dirichlet process) is  not suitable 
in several contexts and in particular in the analysis of mild solutions of an evolution equations that cannot be expected to be in general neither a semimartingale nor a finite quadratic variation process.
A way to remain in this spirit is to introduce a notion of quadratic
variation which is associated with a space (called Chi-subspace) $\chi$
of the dual of the tensor product  $B \hat \otimes_\pi B$. In the rare cases
when the process has indeed a finite  quadratic variation then
the corresponding $\chi$ would be allowed to be the full space
$(B \hat \otimes_\pi B)^*$.

We recall that, following \cite{DiGirolamiRusso09,  DiGirolamiRusso11}, a \emph{Chi-subspace} (of $(H\hat\otimes_\pi G)^*$) is defined as any Banach subspace $(\chi, |\cdot|_\chi)$ which is continuously embedded into $(H\hat\otimes_\pi G)^*$ and, following \cite{FabbriRusso16}, given a Chi-subspace $\chi$ we introduce the notion of $\chi$-covariation as follows.
\begin{Definition}
\label{def:covariation} 
Given two process $\X\colon [s,T] \to H$ and $\X\colon [s,T] \to G$, we say that $(\X, \Y)$ admits a 
{\bf $\chi$-covariation} if the two following conditions are satisfied.
\begin{description}
\item[H1] For any sequence of positive real numbers $\epsilon_{n}\searrow 0$ there exists a subsequence $\epsilon_{n_{k}}$ such that 
\begin{equation} \label{FDefC}
\begin{split}
&\sup_{k}\int_{s}^{T} \frac{\left | J(\X({r+\epsilon_{n_{k}}})-\X({r}))
\otimes (\Y({r+\epsilon_{n_{k}}})-\Y({r}))\right |_{\chi^{\ast}} }{\epsilon_{n_{k}}} dr
\;< \infty\; a.s., 
\end{split}
\end{equation}
where $ J: H\hat{\otimes}_{\pi}G \longrightarrow (H\hat{\otimes}_{\pi}G)^{\ast\ast}$ is the canonical injection between a space and its bidual.

\item[H2]
If we denote by $[\X,\Y]_\chi^{\epsilon}$ the application
\begin{equation}
\label{eq:def-chi-epsilon}
\left \{
\begin{array}{l}
[\X,\Y]_\chi^{\epsilon}:\chi\longrightarrow \mathcal{C}([s,T])\\[5pt]
\displaystyle
\phi \mapsto
\int_{s}^{\cdot} \tensor[_{\chi}]{\left\langle \phi,
\frac{J\left(  \left(\X({r+\epsilon})-\X({r})\right)\otimes \left(\Y({r+\epsilon})-\Y({r})\right)  \right)}{\epsilon} 
\right\rangle}{_{\chi^{\ast}}} dr, 
\end{array}
\right .
\end{equation}
the following two properties hold.
\begin{description}
\item{(i)} There exists an application, denoted by $[\X,\Y]_\chi$, defined on $\chi$ with values in $\mathcal{C}([s,T])$, 
satisfying\footnote{Given a separable Banach space $B$ and a probability space $(\Omega, \mathbb{P})$ a family of processes $\Z^\epsilon\colon \Omega \times [0, T] \to B$ is said to converge in the ucp (uniform convergence on probability) sense to $\Z\colon \Omega \times [0, T] \to B$, when $\epsilon$ goes to zero,
if $\lim_{\epsilon \to 0}  \sup_{t\in [0,T]} |\Z^\epsilon_t - \Z_t|_B = 0$ in probability i.e. if, for any $\gamma>0$, $\lim_{\epsilon \to 0} \mathbb{P} \left (\sup_{t\in [0,T]} |\Z^\epsilon_t - \Z_t|_B >\gamma \right ) = 0$.}
\begin{equation}
[\X,\Y]_\chi^{\epsilon}(\phi)\xrightarrow[\epsilon\longrightarrow 0_{+}]{ucp} [\X,\Y]_\chi(\phi), 
\end{equation} 
for every $\phi \in \chi\subset
(H\hat{\otimes}_{\pi}G)^{\ast}$.
\item{(ii)} 
There exists a Bochner measurable process
 $\widetilde{[\X,\Y]}_\chi:\Omega\times [s,T]\longrightarrow \chi^{\ast}$, 
such that
\begin{itemize}
\item for almost all 
$\omega \in \Omega$, $\widetilde{[\X,\Y]}_\chi(\omega,\cdot)$ is a (c\`adl\`ag) bounded variation process, 
\item 
$\widetilde{[\X,\Y]}_\chi(\cdot,t)(\phi)=[\X,\Y]_\chi(\phi)(\cdot,t)$ a.s. for all $\phi\in \chi$, $t\in [s,T]$.
\end{itemize}
\end{description}
\end{description}
\end{Definition}
If $(\X,\Y)$ admits a {\bf $\chi$-covariation} 
we call  $\widetilde{[\X,\Y]}$ $\chi$-covariation of $(\X,\Y)$.
If  $\widetilde{[\X,\Y]}$ vanishes
we also write  that $[\X,Y]_\chi = 0$.
We say that a process $\X$ admits a {\bf $\chi$-quadratic variation} if $(\X, \X)$
 admits a $\chi$-covariation. In that case  $\widetilde{[\X,\X]}$
is called $\chi$-quadratic variation of $\X$. 


\begin{Definition} \label{DNuDir}
Let $H$ and $G$ be two separable Hilbert spaces.
Let $\nu\subseteq (H\hat\otimes_{\pi} G)^*$ be a Chi-subspace. 
A continuous adapted $H$-valued process $\A \colon [s,T] \times \Omega \to H$ is said to be 
{\bf $\nu$-martingale-orthogonal} if  $[ \A, \N ]_\nu=0$, for any $G$-valued continuous local martingale $\N$.
\end{Definition}
\begin{Lemma}
\label{lm:boundedvar-nu-orthogonal}
Let $H$ and $G$ be two separable Hilbert spaces, $\V\colon [s,T] \times \Omega \to H$ a bounded variation process. \\
For any any Chi-subspace $\nu\subseteq (H\hat\otimes_{\pi} G)^*$, $\V$ is $\nu$-martingale-orthogonal.
\end{Lemma}
\begin{proof}
We will prove that, given any continuous process $\Z\colon [s,T] \times \Omega \to G$  and any Chi-subspace $\nu\subseteq (H\hat\otimes_{\pi} G)^*$, we have $[\V, \Z]_\nu = 0$. This will hold in particular if $\Z$ 
is a continuous local martingale.

By Lemma 3.2 of \cite{FabbriRusso16} it is enough to show that
\[
A(\varepsilon) := \int_s^T \sup_{
\begin{subarray}{c}\Phi\in\nu, \\ \|\Phi\|_\nu\leq 1\end{subarray}
}
\left | \left\langle J \Big ( (\V(t+\varepsilon) - \V(t)) \otimes (\Z(t+\varepsilon) - \Z(t)) \Big ), \Phi \right\rangle \right | \ud t \xrightarrow{\varepsilon \to 0} 0
\]
in probability (the processes are extended on $ ]T,T+\varepsilon]$ by defining, for instance, $\Z(t)=\Z(T)$ for any $t\in ]T,T+\varepsilon]$). Now,
 since $\nu$ is continuously embedded in $(H\hat\otimes_\pi G)^*$, there exists a constant $C$ such that $\|\cdot \|_{(H\hat\otimes_\pi G)^*} \leq C \| \cdot \|_\nu$ so that
\begin{multline}
\label{eq:dinuovolemma}
A(\varepsilon) \leq C \int_s^T \sup_{
\begin{subarray}{c}\Phi\in\nu, \\ \|\Phi\|_{(H\hat\otimes_\pi G)^*}\leq 1\end{subarray}} 
\left | \left\langle J \Big ( (\V(t+\varepsilon) - \V(t)) \otimes (\Z(t+\varepsilon) - \Z(t)) \Big ), \Phi \right\rangle \right | \ud t\\
\leq 
C \int_s^T \left \| J \Big ( (\V(t+\varepsilon) - \V(t)) \otimes (\Z(t+\varepsilon) - \Z(t)) \Big ) \right \|_{(H\hat\otimes_\pi G)^{**}} \ud t \\
= C \int_s^T \left \| \Big ( (\V(t+\varepsilon) - \V(t)) \otimes (\Z(t+\varepsilon) - \Z(t)) \Big ) \right \|_{(H\hat\otimes_\pi G)} \ud t\\
= C \int_s^T \left \|  (\V(t+\varepsilon) - \V(t)) \right \|_H \left \| (\Z(t+\varepsilon) - \Z(t)) \right \|_{G} \ud t,
\end{multline}
where the last step follows by Proposition 2.1 page 16 of \cite{Ryan02}. Now, denoting $t\mapsto |||{\Y}|||(t)$ the real total variation function of an $H$-valued bounded variation function $\Y$ defined on the interval $[s,T]$ we get
\[
\| \Y(t+\varepsilon) - \Y(t) \| = \left \| \int_t^{t+\varepsilon} \ud Y(r) \right \| \leq  \int_t^{t+\varepsilon)} \ud  |||Y||| (r).
\]
So, by using Fubini's theorem in (\ref{eq:dinuovolemma}), 
\[
A(\varepsilon) \leq C \delta(\Z; \varepsilon) \int_s^{T+\varepsilon} \ud |||\V|||(r),
\]
where $\delta(\Z; \varepsilon)$ is the modulus of continuity of $\Z$. Finally this converges to zero almost surely and then in probability.
\end{proof}

\begin{Definition}
\label{def:chi-weak-Dirichlet-process}
Let $H$ and $G$ be two separable Hilbert spaces.
Let $\nu\subseteq (H\hat\otimes_{\pi} G)^*$ be a Chi-subspace. 
A continuous $H$-valued process $\X \colon [s,T] \times \Omega \to H$ is called {\bf $\nu$-weak-Dirichlet process} if it is adapted and there exists a decomposition $\X = \M +\A$ where
\begin{itemize}
 \item[(i)] $\M$ is  an  $H$-valued continuous local martingale,
 \item[(ii)] $\A$ is an $\nu$-martingale-orthogonal process with $\A(s)=0$.
\end{itemize}
\end{Definition}

The theorem below was the object of Theorem 3.19 of \cite{FabbriRusso16}:
it extended Corollary 3.11 in \cite{GozziRusso06}.
\begin{Theorem}
\label{th:prop6}
Let $\nu_0$ be a Banach subspace continuously embedded in $H$. Define $\nu:= \nu_0\hat\otimes_\pi\mathbb{R}$ and $\chi:=\nu_0\hat\otimes_\pi\nu_0$. Let $F\colon [s,T] \times H \to \mathbb{R}$  be a $C^{0,1}$-function.  Denote with $\partial_x F$ the Fr\'echet derivative of $F$ with respect to $x$ and assume that the mapping $(t,x) \mapsto \partial_xF(t,x)$ is continuous from $[s,T]\times H$ to $\nu_0$.  Let $ \X(t) = \M(t) + \A(t)$ for $t\in [s,T]$ be an $\nu$-weak-Dirichlet process with finite $\chi$-quadratic variation. Then $Y(t):= F(t, \X(t))$ is a real weak Dirichlet process
with local martingale part
\[
R(t) = F(s, \X(s)) + \int_s^t \left\langle \partial_xF(r,\X(r)), \ud \M(r) \right\rangle, \qquad t\in [s,T].
\]
\end{Theorem}

\section{The setting of the problem and HJB equation}
\label{sec:HJB}
In this section we introduce a class of infinite dimensional optimal control problems and we prove a 
decomposition result for the strong solutions of the related Hamilton-Jacobi-Bellman equation. We refer the reader to \cite{Yosida80} and \cite{DaPratoZabczyk14} respectively for the classical notions of functional analysis and stochastic calculus in infinite dimension we use.

\subsection{The optimal control problem}
Assume from now that $H$ and $U$ are real separable Hilbert spaces, $Q \in \mathcal{L}(U)$, $U_0:=Q^{1/2} (U)$. Assume that  $\W_Q=\{\W_Q(t):s\leq t\leq T\}$  is an $U$-valued $\mathscr{F}^t_s$-$Q$-Wiener process (with $\W_Q(s)=0$, $\mathbb{P}$ a.s.) and denote by  $\mathcal{L}_2(U_0, H)$ the Hilbert space of the Hilbert-Schmidt operators from $U_0$ to $H$.

We denote  by  $A\colon D(A) \subseteq H \to H$ the generator of the $C_0$-semigroup $e^{tA}$ (for $t\geq 0$) on $H$. $A^*$ denotes the adjoint of $A$. Recall that $D(A)$ and $D(A^*)$ are Banach spaces when endowed with the graph norm. Let $\Lambda$ be a Polish space.

We formulate the following standard assumptions that will be needed to ensure the existence and the uniqueness of the solution of the state equation. 
\begin{Hypothesis}
\label{hp:onbandsigma}
$b\colon [0,T] \times  H \times \Lambda \to H$ is a continuous function and satisfies, for some $C>0$,
\[
\begin{array}{l}
|b(s,x,a) - b(s,y,a)| \leq C |x-y|, \\[3pt]
|b(s,x,a)| \leq C (1+|x|),
\end{array}
\]
for all $x,y \in H$, $s\in [0,T]$, $a\in\Lambda$. $\sigma\colon [0,T]\times H \to \mathcal{L}_2(U_0, H)$ is continuous and, for some $C>0$, satisfies,
\[
\begin{array}{l}
\|\sigma(s,x) - \sigma(s,y)\|_{\mathcal{L}_2(U_0, H)} \leq C |x-y|, \\[3pt]
\|\sigma(s,x)\|_{\mathcal{L}_2(U_0, H)} \leq C (1+|x|),
\end{array}
\]
for all $x,y \in H$, $s\in [0,T]$.
\end{Hypothesis}

\medskip

Given an adapted process $a = a(\cdot): [s,T] \times \Omega \rightarrow \Lambda$, we consider the  state equation
\begin{equation}
\label{eq:state}
\left \{
\begin{array}{l}
\ud \X(t) = \left ( A\X(t)+ b(t,\X(t),a(t)) \right ) \ud t + \sigma(t,\X(t)) \ud \W_Q(t)\\[5pt]
\X(s)=x.
\end{array}
\right.
\end{equation}

The solution of (\ref{eq:state}) is understood in the mild sense: an $H$-valued adapted process $\X(\cdot)$ is a solution if
\[
\mathbb{P} \left \{ \int_s^T \left (|\X(r)| + | b(r,\X(r), a(r))| + \| \sigma(r,\X(r))\|_{\mathcal{L}_2(U_0, H)}^2\right) \ud r <+\infty  \right \} = 1
\]
and
\begin{equation}
\label{eq:state-mild}
\X(t) =  e^{(t-s)A}x + \int_s^t e^{(t-r)A} b(r,\X(r),a(r)) \ud r
 + \int_s^t e^{(t-r)A} \sigma(r,\X(r)) \ud \W_Q(r)
\end{equation}
 $\mathbb{P}$-a.s. for every  $ t \in [s,T]$.
Thanks to Theorem 3.3 of \cite{GawareckiMandrekar10},  given Hypothesis \ref{hp:onbandsigma}, there exists a unique (up to modifications) continuous (mild) solution $\X(\cdot; s,x, a(\cdot))$ of (\ref{eq:state}).

\begin{Proposition}
\label{cor:X-barchi-Dirichlet}
Set $\bar \nu_0 = D(A^*)$, $\nu = \bar \nu_0 \hat\otimes_\pi \mathbb{R}$, $\bar \chi = \bar \nu_0 \hat \otimes_\pi \bar \nu_0.$ 
The process $\X(\cdot; s,x, a(\cdot))$ is $\nu$-weak-Dirichlet process admitting a $\bar \chi$-quadratic variation with decomposition $\M + \A$ where $\M$  is  the local martingale defined by $\M(t) = x + \int_s^t \sigma (r, \X(r)) \ud \W_Q(r)$ and $\A$ is a $\nu$-martingale-orthogonal process.
\end{Proposition}
\begin{proof}
See Corollary 4.6 of \cite{FabbriRusso16}.
\end{proof}

\begin{Hypothesis}
\label{hp:on-l-and-g}
Let $l\colon [0,T] \times H \times \Lambda \to \mathbb{R}$ (the running cost) be a measurable function and $g\colon H\to\mathbb{R}$ (the terminal cost) a continuous function.
\end{Hypothesis}

\smallskip
 
We consider the class $\mathcal{U}_s$ of admissible controls constituted by
the adapted processes $a:[s,T] \times \Omega \rightarrow \Lambda$ such that $(r, \omega) \mapsto l(r, \X(r,s,x, a(\cdot)), a(r)) + 
g(\X(T, s, x, a(\cdot))) $ is $\ud r \otimes \ud \P$- is 
quasi-integrable. This means that, either  its positive or negative part are integrable. \\

\smallskip

We consider the problem of minimizing, over all $a(\cdot) \in \mathcal{U}_s$, the  cost functional
\begin{equation}
\label{eq:functional}
J(s,x;a(\cdot))=\mathbb{E}\bigg[ \int_{s}^{T} l(r, \X(r;s,x,a(\cdot)), a(r)) \ud r + g (\X(T;s,x,a(\cdot))) \bigg].
\end{equation}
The value function of this problem is defined, as usual, as
\begin{equation}
\label{eq:def-valuefunction}
V(s,x) = \inf_{a(\cdot)\in \mathcal{U}_s} J(s,x;a(\cdot)).
\end{equation}

As usual we say that the control $a^*(\cdot)\in \mathcal{U}_s$ is \emph{optimal} at $(s,x)$ if $a^*(\cdot)$ minimizes (\ref{eq:functional}) among the controls in $\mathcal{U}_s$, i.e. if $J(s,x;a^*(\cdot)) = V(s,x)$. In this case we denote by $\X^*(\cdot)$ the process $\X(\cdot; s,x, a^*(\cdot))$ which is then the corresponding \emph{optimal trajectory} of the system.



\subsection{The HJB equation}

The HJB equation associated with the minimization problem above is
\begin{equation}
\label{eq:HJB}
\left\{
\begin{array}{l}
\partial_s v + \left \langle  A^* \partial_x v, x \right\rangle + \frac{1}{2} Tr \left [ \, \sigma(s,x) \sigma^*(s,x) \partial_{xx}^2 v \right ] \\[3pt]
\qquad\qquad\qquad\qquad  + \inf_{a\in \Lambda }\Big\{  \left \langle \partial_x v, b(s,x,a) \right\rangle + l(s,x,a) \Big\}=0,\\ [8pt]
v(T,x)=g(x).
\end{array}
\right.
\end{equation}
In the above equation $\partial_xv$ (respectively $\partial^2_{xx} v$) is the first
(respectively second) Fr\'echet 
derivatives of $v$ with respect to the $x$ variable.
Let $(s,x) \in [0,T] \times H$,
$\partial_xv(s,x)$  it is identified (via Riesz Representation Theorem, see \cite{Yosida80}, Theorem III.3) with elements of $H$.
 $\partial^2_{xx} v(s,x)$ which is a priori an element of
$(H \hat \otimes_\pi H)^*$ is naturally  associated 
with a symmetric bounded operator on $H$,
see \cite{Flett80},  statement 3.5.7, page 192.
In particular, if $h_1,h_2 \in H$ then 
 $\langle \partial^2_{xx} v(s,x), h_1 \otimes h_2 \rangle 
\equiv  \partial^2_{xx} v(s,x)(h_1)(h_2)$.
 $\partial_s v$ is the derivative with respect to the time variable.
\\
The function
\begin{equation}
\label{eq:def-CV-Hamiltonian}
F_{CV}(s,x,p;a):=  \left \langle p, b(s,x,a) \right\rangle + l(s,x,a), \quad (s,x,p,a)\in [0,T]\times H \times H\times \Lambda,
\end{equation}
is called the \emph{current value Hamiltonian} of the system and its infimum over $a\in \Lambda$ 
\begin{equation}
\label{eq:def-Hamiltonian}
F(s,x,p):= \inf_{a\in \Lambda }\left\{  \left \langle p, b(s,x,a) \right\rangle + l(s,x,a) \right\}
\end{equation}
is called the \emph{Hamiltonian}.
We remark that $F:[0,T] \times H \times H \rightarrow [-\infty +\infty[$. 
 Using this notation the HJB equation (\ref{eq:HJB}) 
can be rewritten as
\begin{equation}
\label{eq:HJB-with-F}
\left\{
\begin{array}{l}
\partial_s v+  \left \langle A^* \partial_x v, x \right\rangle + \frac{1}{2} Tr \left [ \sigma(s,x) \sigma^*(s,x) \partial^2_xv \right ] + F(s,x,\partial_xv)=0,\\[6pt]
v(T,x)=g(x).
\end{array}
\right.
\end{equation}


We introduce the operator $\mathscr{L}_0$ on $C([0,T]\times H)$ defined as
\begin{equation}
\label{eq:def-L0}
\left \{
\begin{array}{l}
D(\mathscr{L}_0):= \left \{ \varphi \in C^{1,2}([0,T]\times H) \; : \; \partial_x\varphi \in C([0,T]\times H ; D(A^*)) \right \} \\[8pt]
\mathscr{L}_0 (\varphi)(s,x) := \partial_s \varphi(s,x)+  \left \langle A^* \partial_x \varphi (s,x), x \right\rangle + \frac{1}{2} Tr \left [  \sigma(s,x) \sigma^*(s,x) \partial_{xx}^2 \varphi(s,x) \right ],
\end{array}
\right .
\end{equation}
so that the HJB equation (\ref{eq:HJB-with-F}) can be formally rewritten as 
%
\begin{equation}
\label{eq:simil-HJB-con-h}
\left \{
\begin{array}{l}
\mathscr{L}_0 (v) (s,x) = - F(s,x, \partial_xv(s,x)) \\[5pt] 
v(T,x) = g(x).
\end{array}
\right .
\end{equation}

Recalling that we suppose the validity of Hypothesis \ref{hp:on-l-and-g} 
we consider the two following definitions of solution of the HJB equation.
\begin{Definition}
\label{def:classical-sol}
We say that $v \in C([0,T]\times H)$ is a {\bf classical solution} of (\ref{eq:simil-HJB-con-h}) if 
\begin{itemize}
 \item[(i)] $v\in D(\mathscr{L}_0)$
 \item[(ii)] The function  
 \[
 \left \{
\begin{array}{l}
\left[ 0,T\right] \times H \to \mathbb{R}\\
(s,x) \mapsto F(s,x,\partial_xv(s,x))
\end{array}
\right .
 \]
is well-defined and finite for all $\left( s,x\right) \in \left[ 0,T\right] \times H$ and it is continuous in the two variables
\item[(iii)] (\ref{eq:simil-HJB-con-h}) is satisfied at any $\left( s,x\right) \in \left[ 0,T\right] \times H$.
\end{itemize}
\end{Definition}

\begin{Definition}
\label{def:strong-sol}
Given $g\in C(H)$ we say that $v  \in C^{0,1}([0,T[ \times H)  \cap  C^{0}([0,T] \times H)$ with  $\partial_x v \in UC([0,T[\times H; D(A^*))$
\footnote{The space of uniformly
continuous functions on each ball of $[0,T[ \times H$ with values
 in $D(A^*)$.} 
 is a {\bf strong solution} of (\ref{eq:simil-HJB-con-h}) if the following properties hold.
\begin{itemize}
 \item[(I)] The function $(s,x) \mapsto F(s,x,\partial_xv(s,x))$ is finite
 for all $\left(s,x\right) \in \left[ 0,T\right [ \times H$, it is continuous in the two variables
and admits continuous extension on  $\left[ 0,T\right ] \times H$.
\item[(II)] There exist three sequences $ \{v_n \} \subseteq D(\mathscr{L}_0)$, $\{h_n\} \subseteq C([0,T]\times H)$ and $ \{ g_n \} \subseteq C(H)$ fulfilling the following.
\begin{enumerate}
\item[(i)] For any $n\in\mathbb{N}$, $v_n$ is a classical solution of the problem
\begin{equation}
\label{eq:approximating}
\left \{
\begin{array}{l}
\mathscr{L}_0 (v_n) (s,x) = h_n(s,x)\\[5pt]
v_n(T,x) = g_n(x).
\end{array}
\right .
\end{equation}
\item[(ii)] The following convergences hold:
\[
\left \{
\begin{array}{ll}
v_n\to v & \text{in} \;\; C([0,T]\times H)\\
h_n\to - F(\cdot,\cdot, \partial_xv(\cdot,\cdot)) & \text{in} \;\; C([0,T]\times H)\\
g_n\to g & \text{in} \;\; C(H),
\end{array}
\right .
\]
where the convergences in $C([0,T]\times H)$ and $C(H)$ are meant in the sense of uniform convergence on compact sets.
\end{enumerate}
\end{itemize}
\end{Definition}

\begin{Remark}
\label{rm:comparison-definition-literature}
The notion of \emph{classical solution} as defined in Definition \ref{def:classical-sol} is well established in the literature of second-order infinite dimensional Hamilton-Jacobi equations, see for instance Section 6.2 of \cite{DaPratoZabczyk02}, page 103. Conversely the denomination \emph{strong solution} is used for a certain number of definitions where the solution of the Hamilton-Jacobi equation is characterized by the existence of a certain approximating sequence (having certain properties and) converging to the candidate solution. The chosen functional spaces and the prescribed convergences depend on the classes of equations, see for instance \cite{BarbuDaPrato83book, Cerrai01, Gozzi96, Gozzi02, Priola99}. In this sense the solution defined in Definition \ref{def:strong-sol} is a form of strong solution of (\ref{eq:simil-HJB-con-h}) but, differently to all other papers we know\footnote{Except \cite{GozziRusso06Stoch}, but there the HJB equation and the optimal controls are finite dimensional.} we do not require any form of convergence of the derivatives of the approximating functions to the derivative of the candidate solution. Moreover all the results we are aware of use sequences of bounded approximating functions (i.e. the $v_n$ in the definition are bounded) and this is not required in our definition. All in all the sets of approximating sequences that we can manage are bigger than those used in the previous literature and so the definition of strong solution is weaker.
\end{Remark}

\color{black}

\subsection{Decomposition for solutions of the HJB equation}

\begin{Theorem}
\label{th:decompo-sol-HJB}
Suppose Hypothesis \ref{hp:onbandsigma} is satisfied. 
Suppose that $v  \in C^{0,1}([0,T[ \times H)  \cap  C^{0}([0,T] \times H)$ with $\partial_x v \in UC([0,T[\times H; D(A^*))$
  is a strong solution of (\ref{eq:simil-HJB-con-h}). Let $\X(\cdot):=\X(\cdot;t,x,a(\cdot))$ be the solution of (\ref{eq:state}) starting at time $s$ at some $x\in H$ and driven by some control $a(\cdot)\in \mathcal{U}_s$.
Assume that $b$ is of the form 
\begin{equation}\label{eq:b}
b(t,x,a) = b_g(t,x,a) + b_i(t,x,a),
\end{equation}
where $b_g$ and $b_i$ satisfy the following conditions.
\begin{itemize}
\item[(i)] $\sigma(t,\X(t))^{-1} b_g(t,\X(t),a(t))$ is bounded (being $\sigma(t,\X(t))^{-1}$ the pseudo-inverse of $\sigma$);
\item[(ii)] $b_i$ satisfies 
\begin{equation}
\label{eq:conv-ucp-b}
\lim_{n\to\infty} \int_s^\cdot \left \langle \partial_x v_n (r,\X(r)) - \partial_x v (r,\X(r)), b_i(r,\X(r), a(r)) \right\rangle \ud r =0 \quad \text{ucp on $[s,T_0]$},
\end{equation}
for each $s < T_0 <T$. 
\end{itemize}
Then
\begin{multline} \label{E86}
v(t, \X(t)) - v(s, \X(s)) =  v(t, \X(t)) - v(s, x) = - \int_s^t F(r,\X(r), \partial_xv(r,\X(r))) \ud r \\
+ \int_s^t \left\langle \partial_x v(r, \X(r)), b(r,\X(r), a(r)) \right\rangle \ud r
+ \int_s^t \left\langle \partial_x v(r, \X(r)), \sigma (r,\X(r)) \ud \W_Q(r) \right\rangle,  \ t \in [s,T[.
\end{multline}
\end{Theorem}
\begin{proof}
We fix $T_0$ in $]s,T[$.
We  denote  by  $v_n$ the sequence of
smooth solutions of the approximating problems prescribed 
by Definition \ref{def:strong-sol}, which converges to $v$.
Thanks to It\^o formula for convolution type processes (see e.g. Corollary 4.10 in \cite{FabbriRusso16}), every $v_n$ verifies
\begin{multline}
\label{eq:v-Ito-pre}
v_n(t,\X(t)) =  v_n(s,x)  + \int_s^t \partial_r {v_n}(r,\X(r)) \ud r \\ 
 +   \int_s^t \left\langle A^* \partial_x v_n(r,\X(r)), \X(r) \right\rangle \ud r
 +  \int_s^t \left\langle \partial_x v_n(r,\X(r)), b(r, \X(r), a(r)) 
\right\rangle \ud r \\
+ \frac{1}{2}  \int_s^t \text{Tr} \left [\left ( \sigma (r, \X(r)) {Q}^{1/2}
 \right )
\left ( \sigma (r, \X(r)) Q^{1/2}  \right)^* \partial_{xx}^2 
v_n(r,\X(r)) \right ] \ud r  \\
+ \int_s^t \left\langle \partial_x v_n(r,\X(r)), \sigma(r, \X(r)) \ud \W_Q(r)
 \right\rangle, \ t \in [s,T]. \qquad \mathbb{P}-{\rm a.s.}
\end{multline}
Using Girsanov's Theorem (see \cite{DaPratoZabczyk14} Theorem 10.14) we can observe that
\[
\beta_Q(t) := W_Q(t) + \int_s^t \sigma(r,\X(r))^{-1} b_g(r,\X(r), a(r)) \ud r,
\]
is a $Q$-Wiener process with respect to
 a probability $\mathbb{Q}$ equivalent to $\mathbb{P}$ on the whole interval $[s,T]$.
 We can rewrite (\ref{eq:v-Ito-pre}) as
\begin{multline}
\label{eq:v-Ito}
v_n(t,\X(t)) =  v_n(s,x)  + \int_s^t \partial_r {v_n}(r,\X(r)) \ud r\\  +  \int_s^t \left\langle A^* \partial_x v_n(r,\X(r)), \X(r) \right\rangle \ud r
 +  \int_s^t \left\langle \partial_x v_n(r,\X(r)), b_i(r, \X(r), a(r)) \right\rangle \ud r,\\
+ \frac{1}{2}  \int_s^t \text{Tr} \left [\left ( \sigma (r, \X(r)) {Q}^{1/2} \right )
\left ( \sigma (r, \X(r)) Q^{1/2}  \right)^* \partial_{xx}^2 v_n(r,\X(r)) \right ] \ud r\\
+ \int_s^t \left\langle \partial_x v_n(r,\X(r)), \sigma(r, \X(r)) \ud \beta_Q(r) \right\rangle. \qquad \mathbb{P}-a.s.
\end{multline}
Since $v_n$ is a classical solution of (\ref{eq:approximating}), the expression
above gives
\begin{multline}
v_n(t,\X(t)) =  v_n(s,x)  + \int_s^t {h_n}(r,\X(r)) \ud r\\ +  \int_s^t \left\langle \partial_x v_n(r,\X(r)), b_i(r, \X(r), a(r)) \right\rangle \ud r
+ \int_s^t \left\langle \partial_x v_n(r,\X(r)), \sigma(r, \X(r)) \ud \beta_Q(r)\right\rangle.
\end{multline}
Since we wish to take the limit for $n\to\infty$, we define
\begin{multline}
M_n(t) := v_n(t,\X(t)) -  v_n(s,x)  - \int_s^t {h_n}(r,\X(r)) \ud r \\
-  \int_s^t \left\langle \partial_x v_n(r,\X(r)), b_i(r, \X(r), a(r)) \right\rangle \ud r.
\end{multline}
$\{ M_n \}_{n\in\mathbb{N}}$ is a sequence of real $\mathbb{Q}$-local
 martingales converging ucp, thanks to the definition of strong solution
 and Hypothesis (\ref{eq:conv-ucp-b}), to 
\begin{multline} \label{E72}
M(t) := v(t,\X(t)) -  v(s,x)  + \int_s^t F(r,\X(r), \partial_xv(r,\X(r))) \ud r \\
-  \int_s^t \left\langle \partial_x v(r,\X(r)), b_i(r, \X(r), a(r)) \right\rangle \ud r, \ t \in [s,T_0].
\end{multline} 
Since the space of real continuous local martingales equipped with 
the ucp topology is closed (see e.g. Proposition 4.4 of \cite{GozziRusso06}) 
then $M$ is a continuous $\mathbb{Q}$-local martingale indexed by
$t \in [s,T_0]$.

\medskip
We have now gathered all the ingredients to conclude the proof.
 We set $\bar \nu_0 = D(A^*)$, 
$\nu = \bar \nu_0 \hat\otimes_\pi \R, \bar \chi = \bar \nu_0 \hat \otimes_\pi \bar \nu_0$. Proposition \ref{cor:X-barchi-Dirichlet} ensures that $\X(\cdot)$ is 
a $\nu$-weak Dirichlet process admitting a $\bar \chi$-quadratic variation  with decomposition $\M + \A$ where $\M$  is  the local martingale  (with respect to $\P$) defined by $\M(t) = x + \int_s^t \sigma (r, \X(r)) \ud \W_Q(r)$ and $\A$ is a $\nu$-martingale-orthogonal process.
Now 
$$\X(t) = \tilde \M(t) 
 + \V(t) + \A(t),
 t \in [s,T_0],$$ 
where $\tilde \M(t) =  x + \int_s^t \sigma (r, \X(r)) \ud \beta_Q(r)$
and  $\V(t) = - \int_s^t b_g(r, \X(r),a(r)) dr$, $t \in [s,T_0]$,
is a bounded variation process. Thanks to \cite{KrylovRozovskii07} Theorem 2.14 page 14-15, $\tilde \M$ is a $\Q$-local martingale. Moreover 
$\V$ is a bounded variation process and then, thanks to Lemma  
\ref{lm:boundedvar-nu-orthogonal}, it is a $\Q-\nu$-martingale orthogonal process.
 So $\V + \A$ is a again (one can easily verify that the sum of two $\nu$-martingale-orthogonal processes is again a $\nu$-martingale-orthogonal process) a $\Q-\nu$-martingale orthogonal process and $\X$ is a $\nu$-weak Dirichlet process with  local martingale part $\tilde \M$, with respect to $\Q$.
Still under $\Q$, since $v \in C^{0,1}([0,T_0] \times H)$, 
Theorem \ref{th:prop6} 
 ensures that the process $v(\cdot, \X(\cdot))$ is a real
 weak Dirichlet process on $[s,T_0]$,
 whose local martingale part being equal to
\[
N(t) = \int_s^t \left\langle \partial_x v(r,\X(r)), \sigma(r, \X(r)) \ud \beta_Q(r)\right\rangle, t \in [s,T_0]. 
\]
On the other hand, with respect to $\mathbb Q$, \eqref{E72} implies that 
\begin{multline}
v(t,\X(t)) = \bigg [  v(s,x)  - \int_s^t F(r,\X(r), \partial_xv(r,\X(r))) \ud r \\
+  \int_s^t \left\langle \partial_x v(r,\X(r)), b_i(r, \X(r), a(r)) \right\rangle \ud r \bigg ] + N(t), \ t \in [s,T_0],
\end{multline}
is a decomposition of $v(\cdot,\X(\cdot))$ as $\mathbb Q$-
semimartingale, which is also in particular, a  $\mathbb Q$-weak Dirichlet process.
 By Theorem \ref{th:uniqueness-weak-Dirichlet} such a decomposition is unique 
on $[s,T_0]$  and so
$ M(t) = N(t), t \in [s,T_0]$, so $ M(t) = N(t), t \in [s,T[$.



Consequently 
  \begin{multline}
 M(t)  = \int_s^t \left\langle \partial_x v(r,\X(r)), \sigma(r, \X(r)) \ud \beta_Q(r)\right\rangle\\ 
 = \int_s^t \left\langle \partial_x v(r,\X(r)), b_g(r, \X(r),a(r)) \ud r\right\rangle\\
 + \int_s^t \left\langle \partial_x v(r,\X(r)), \sigma(r, \X(r))
  \ud \W_Q(r)\right\rangle, \ t \in [s,T].
 \end{multline}
\end{proof}

\begin{Example} \label{EThSol-HJB}
The decomposition \eqref{eq:b} with validity of
Hypotheses $(i)$ and $(ii)$ in Theorem \ref{th:decompo-sol-HJB} are satisfied if $v$ is a strong solution of the HJB equation in the sense of Definition \ref{def:strong-sol} and, moreover the sequence of corresponding functions
 $\partial_x v_n$ converge to $\partial_xv$ in $C([0,T]\times H)$. 
In that case we simply set $b_g = 0$ and $b = b_i$.
 This is the typical assumption required in the standard strong solutions literature. 
\end{Example}
\begin{Example} \label{EThSol-HJB1}
Again the decomposition \eqref{eq:b} with validity of
Hypotheses $(i)$ and $(ii)$ in Theorem \ref{th:decompo-sol-HJB} 
 is fulfilled
 if the following assumption is satisfied.
\[
\sigma(t,\X(t))^{-1} b(t,\X(t),a(t)) \text{ is bounded},
\]
for all choice of admissible controls $a(\cdot)$.
In this case we apply Theorem  \ref{th:decompo-sol-HJB}  
with $b_i = 0$  and $b=b_g$.
\end{Example}

\section{Verification Theorem}
\label{sec:verification}

In this section, as anticipated in the introduction, we use the decomposition result of Theorem \ref{th:decompo-sol-HJB} to prove a verification theorem.

\begin{Theorem}
\label{th:verification}
Assume that Hypotheses \ref{hp:onbandsigma} and \ref{hp:on-l-and-g} 
are satisfied and that the value function is finite for any $(s,x)\in [0,T]\times H$. 
Let  $v  \in C^{0,1}([0,T[ \times H)  \cap  C^{0}([0,T] \times H)$
with $\partial_x v \in UC([0,T[\times H; D(A^*))$  be a
 strong solution of (\ref{eq:HJB}) and suppose that there exists two constants $M>0$ and $m\in \mathbb{N}$ such that 
 $| \partial_xv(t,x)| \leq M(1+|x|^{m})$ for all $(t,x)\in [0,T[\times H$.\\
Assume that for all initial data
 $(s,x)\in [0,T]\times H$ and every control $a(\cdot)\in \mathcal{U}_s$ $b$
 can be written as $b(t,x,a) = b_g(t,x,a) + b_i(t,x,a)$ with $b_i$ and $b_g$
 satisfying hypotheses (i) and (ii) of Theorem \ref{th:decompo-sol-HJB}.
 Then
we have the following.
\begin{enumerate}
\item[(i)]  $v\leq V$ on $[0,T]\times H$.
\item[(ii)]  Suppose that, for some $s\in [0,T[$, 
there exists a 
predictable process $a(\cdot) = a^*(\cdot)  \in \mathcal{U}_s$ 
such that, denoting $\X\left( \cdot;s,x,a^*(\cdot)\right)$ simply
 by $\X^*(\cdot)$,
we have
\begin{equation}
\label{eq:condsuff}
F\left( t, \X^*\left( t\right) ,\partial_xv\left( t,\X^*\left( t\right)
\right) \right) =F_{CV}\left( t,\X^*\left( t\right) ,\partial_xv\left(
t,\X^*\left( t\right) \right) ;a^*\left( t\right) \right),
\end{equation}
  $dt \otimes \ud \P$ a.e. 
Then  $a^*(\cdot)$ is optimal at $\left( s,x\right) $; moreover 
$v\left( s,x\right) =V\left(s,x\right)$.
\end{enumerate}
\end{Theorem}
\begin{proof}
We choose a control $a(\cdot) \in \mathcal{U}_s$  and call $\X$ the related trajectory. We make use of  \eqref{E86} in Theorem
 \ref{th:decompo-sol-HJB}. 
Then we need to extend \eqref{E86} to the case
when $t \in [s,T]$.
 This is possible since $v$ is continuous, 
$(s,x) \mapsto F(s,x,\partial_xv(s,x))$ is well-defined and (uniformly continuous) on compact sets.
At this point, setting $t = T$ 
 we can write
 \begin{multline}
\label{ex66bis}
g(\X(T)) = v(T, \X(T)) = v(s, x) - \int_s^T F (r, \X(r), \partial_xv(r,\X(r))) \ud r \\
+ \int_s^T \left\langle \partial_x v(r, \X(r)), b(r,\X(r), a(r)) \right\rangle \ud r
+ \int_s^T \left\langle \partial_x v(r, \X(r)), \sigma (r,\X(r)) \ud \W_Q(r) \right\rangle.
\end{multline}
Since both sides of \eqref{ex66bis} are a.s. finite, we can add 
 $\int_s^T l(r, \X(r) , a(r)) \ud r$ to them, obtaining 
\begin{multline}
\label{eq:ex66}
g(\X(T)) + \int_s^T l(r, \X(r) , a(r)) \ud r 
 = v(s, x) + \int_s^T \left\langle \partial_x v(r, \X(r)), \sigma (r,\X(r)) \ud \W_Q(r) \right\rangle \\
+ \int_s^T \left(- F (r, \X(r), \partial_xv(r,\X(r))) +
 F_{CV} (r, \X(r), \partial_xv(r,\X(r));a(r))\right) \ud r.
\end{multline}
Observe now that, by definition of $F$ and $F_{CV}$ we know that 
\[
- F (r, \X(r), \partial_xv(r,\X(r))) + F_{CV} (r, \X(r), \partial_xv(r,\X(r))
;a(r))
\]
is always positive. So its expectation always exists even if it 
could be $+\infty$, but not $-\infty$ on an event of positive probability.
This shows a posteriori that 
$ \int_s^T l(r, \X(r) , a(r)) \ud r $ cannot be $-\infty$
on a set of positive probability. \\
By Proposition 7.4 in \cite{DaPratoZabczyk14}, all the momenta of 
$\sup_{r \in [s,T]} \vert \X(r) \vert $ are 
finite.
On the other hand, $\sigma$ is Lipschitz-continuous,
  $v(s,x)$ is deterministic and, 
 since $\partial_x v$ has polynomial growth, 
  then
$$ \mathbb{E} \int_s^T \left\langle \partial_x v(r, \X(r)), \left ( \sigma (r,\X(r)) Q^{1/2} \right ) \left ( \sigma (r,\X(r)) Q^{1/2} \right )^* \partial_x v(r, \X(r)) \right\rangle \ud r
$$
is finite. Consequently (see \cite{DaPratoZabczyk14} Sections 4.3, in particular Theorem 4.27 and 4.7),
\[
 \int_s^\cdot \left\langle \partial_x v(r, \X(r)), \sigma (r,\X(r)) \ud
 \W_Q(r) \right\rangle,
\]
is  a true martingale vanishing at $s$.  Consequently,
 its expectation is zero. 
 So the expectation of the right-hand side of (\ref{eq:ex66})
 exists even if it could be $+\infty$; consequently the same holds for the left-hand side.\\
By definition of $J$, we have
\begin{multline}
\label{eq:finale}
J(s,x,a(\cdot)) = \mathbb{E} \bigg [ g(\X(T)) + \int_s^T l(r, \X(r) , a(r)) \ud r \bigg ] = v(s, x) \\
+ \mathbb{E} \int_s^T  \Big ( - F (r, \X(r), \partial_xv(r,\X(r))) + F_{CV}(r, \X(r), \partial_xv(r,\X(r)); a(r)) \Big ) \ud r.
\end{multline}
So minimizing $J(s,x,a(\cdot))$ over $a(\cdot)$ is equivalent to minimize
\begin{equation}
\label{eq:ultima-2}
\mathbb{E} \int_s^T  \Big ( - F (r, \X(r), \partial_xv(r,\X(r))) + F_{CV}(r, \X(r), \partial_xv(r,\X(r)); a(r)) \Big )\ud r,
\end{equation}
which is a non-negative quantity.
As mentioned above, the integrand of such an expression is always nonnegative and then a lower bound for (\ref{eq:ultima-2}) is $0$. If the conditions of point (ii) are satisfied such a bound is attained by the control $a^*(\cdot)$, that in this way is proved to be optimal.

Concerning the proof of (i), since the integrand in (\ref{eq:ultima-2}) is nonnegative,  (\ref{eq:finale}) gives
\[
J(s,x,a(\cdot)) \geq  v(s, x).
\]
Taking the inf over $a(\cdot)$ we get $V(s,x) \geq v(s,x)$, which concludes the proof.
\end{proof}

\begin{Remark} \label{RFeedback}
\begin{enumerate}
\item The first part of the proof does not make use that
$a$ belongs to ${\mathcal U}_s$, but only that
$r \mapsto l(r,\X(\cdot,s,x,a(\cdot)), a(\cdot))$ 
is a.s. strictly bigger then $-\infty$.  Under  that only assumption, 
$a(\cdot)$ is forced to  be admissible, i.e. to belong to
${\mathcal U}_s$.
\item Let $v$ be a strong solution of HJB equation.
Observe that the condition (\ref{eq:condsuff}) can be rewritten as
\[
a^*(t) \in \arg\min_{a\in \Lambda} \Big [ F_{CV} \left( t,\X^*\left( t\right), 
\partial_xv\left(
t,\X^*\left( t\right) \right) ;a \right) \Big ]
.\]
Suppose the existence of a Borel function $\phi:[0,T] \times H \rightarrow \R$
such that 
 for any $(t,y) \in [0,T] \times H $,
 $\phi(t,y) \in  \arg\min_{a \in \Lambda} \big  ( F_{CV}\left(t,y, 
 \partial_xv(t,y);a\right) \big )$. 

Suppose that the equation 
\begin{equation}
\label{eq:stateFeed}
\left \{
\begin{array}{l}
\ud \X(t) = \left ( A\X(t)+ b(t,\X(t),\phi(t,\X(t) \right ) \ud t + \sigma(t,\X(t)) \ud \W_Q(t)\\[5pt]
\X(s)=x,
\end{array}
\right.
\end{equation}
admits a unique mild solution $\X^*$.
We set $a^*(t) = \phi(t, \X^*(t)), t \in [0,T]$.
 Suppose moreover that 
 \begin{equation}  \label{EFeedAdm}
 \int_s^T l(r,\X^*(r), a^*(r)) dr > - \infty \ {\rm a.s.}
 \end{equation}

Now  \eqref{EFeedAdm} and Remark \ref{RFeedback} 1.
imply that $a^*(\cdot)$ is admissible.
Then $\X^*$ is the optimal trajectory of the state variable related to the optimal control $a^*(t)$.
 The function $\phi$ is
called  \emph{optimal feedback} of the system since it gives 
an optimal control as a function of the state.
\end{enumerate}
\end{Remark}

\begin{Remark} 
Observe that, using exactly the same arguments we used in this section one could treat the (slightly) more general case in which $b$ has the form
\[
b(t,x,a)= b_0(t,x) + b_g(t,x,a) + b_i(t,x,a),
\]
where $b_g$ and $b_i$ satisfy condition of Theorem \ref{th:decompo-sol-HJB}
and $b_0: [0,T] \times H \rightarrow H$ is continuous. In this case the addendum $b_0$ can be included in the expression of $\mathscr{L}_0$ that becomes 
\begin{equation}
\label{eq:def-L0-b}
\left \{
\begin{array}{l}
D(\mathscr{L}_0^{b_0}):= \left \{ \varphi \in C^{1,2}([0,T]\times H) \; : \; \partial_x\varphi \in C([0,T]\times H ; D(A^*)) \right \} \\[6pt]
\mathscr{L}_0^{b_0} (\varphi) := \partial_s \varphi+  \left \langle A^* \partial_x \varphi, x \right\rangle + \left \langle \partial_x \varphi, b_0(t,x) \right\rangle + \frac{1}{2} Tr \left [ \sigma(s,x) \sigma^*(s,x) \partial_{xx}^2 \varphi \right ].
\end{array}
\right .
\end{equation}
Consequently in the definition of regular solution the operator $\mathscr{L}_0^{b_0}$ appears instead $\mathscr{L}_0$.
\end{Remark}

\section{An example}
\label{sec:example}
We describe in this section an example where the techniques developed in the previous part of the paper can be applied. It is rather simple but some ``missing'' regularities and continuities show up so that it cannot be treated by using the standard techniques (for more details see Remark \ref{rm:discussion-example}).

Denote by $\Theta\colon \mathbb{R} \to \mathbb{R}$ the Heaviside function
\[
\left \{
\begin{array}{l}
\Theta\colon \mathbb{R} \to \mathbb{R}\\
\Theta\colon y\mapsto 
\left\{
\begin{array}{ll}
1 & \text{if } y \geq 0\\
0 & \text{if } y < 0.
\end{array}
\right .
\end{array}
\right .
\]

Fix $T>0$. Let $\rho,\beta$ be two real numbers, $\psi\in D(A^*) \subseteq H$ an eigenvector\footnote{Examples where the optimal control distributes as an eigenvector of $A^*$ arise in applied examples, see for instance \cite{boucekkine17growth,Fabbri16} for some economic deterministic examples. In the mentioned cases the operator $A$ is elliptic and self-adjoint.} for the operator $A^*$ corresponding to an eigenvalue $\lambda\in \mathbb{R}$, $\phi$ an element of $H$ and $W$ a standard real (one-dimensional) Wiener process. We consider the case where $\Lambda = \mathbb{R}$ (i.e. we consider real-valued controls). Let us take into account a state equation of the following specific form:
\begin{equation}
\label{eq:state-example}
\left \{
\begin{array}{l}
\ud \X(t) = \left ( A\X(t)+ a(t)\phi \right ) \ud t + \beta\X(t) \ud W(t)\\[5pt]
\X(s)=x.
\end{array}
\right .
\end{equation}
The operator $\mathscr{L}_0$ specifies then as follows:
\begin{equation}
\label{eq:def-L0-example}
\left \{
\begin{array}{l}
D(\mathscr{L}_0):= \left \{ \varphi \in C^{1,2}([0,T]\times H) \; : \; \partial_x\varphi \in C([0,T]\times H ; D(A^*)) \right \} \\[8pt]
\mathscr{L}_0 (\varphi)(s,x) := \partial_s \varphi(s,x)+  \left \langle A^* \partial_x \varphi (s,x), x \right\rangle + \frac{1}{2}\beta^2 \left \langle x, \partial_{xx}^2 \varphi(s,x)(x)\right\rangle.
\end{array}
\right .
\end{equation}
Denote by $\alpha$ the real constant $\alpha := \frac{-\rho + 2\lambda + 
\beta^2}{\left\langle \phi, \psi \right\rangle^2}$. We take into account the 
 functional
\begin{multline}
\label{eq:functional-example}
J(s,x;a(\cdot))=\mathbb{E}\bigg[ \int_{s}^{T} e^{-\rho r} \Theta \left ( \left \langle \X(r;s,x,a(\cdot)), \psi \right\rangle \right ) a^2(r) \ud r \\
+ e^{-\rho T} \alpha \Theta \left ( \left \langle \X(T;s,x,a(\cdot)), \psi \right\rangle \right ) \left \langle \X(T;s,x,a(\cdot)), \psi \right\rangle^2 \bigg].
\end{multline}
The Hamiltonian associated to the problem is given by
\begin{equation}
\label{eq:Hamiltonian-example}
F(s,x,p):= \inf_{a\in \mathbb{R} } F_{CV}(s,x,p;a),
\end{equation}
where 
$$ F_{CV}(s,x,p;a) =
 \langle p, a \phi \rangle +
 e^{-\rho s} \Theta \left (\langle x, \psi \right\rangle) a^2.
$$
Standard calculations give
\begin{equation} \label{E40}
F(s,x,p) = \left \{
\begin{array}{ccc}
- \infty &:& p \neq 0, \langle x, \psi \rangle < 0 \\
- \frac{\langle p, \phi \rangle^2}{4} &:& {\rm otherwise}.
\end{array}
\right.
\end{equation}
The HJB equation is
\begin{equation}
\label{eq:HJB-example}
\left\{
\begin{array}{l}
\mathscr{L}_0(v)(s,t) = - F(s,x,\partial_{x} v (s,x)),\\ [8pt]
v(T,x)=g(x) := e^{-\rho T} \alpha \Theta \left ( \left \langle x, \psi \right\rangle \right ) \left \langle x, \psi \right\rangle^2.
\end{array}
\right.
\end{equation}

\begin{Lemma}
The function
\begin{equation}
\label{eq:defv-example}
\left \{
\begin{array}{l}
v\colon [0,T] \times H \to \mathbb{R}\\ 
v \colon (s,x) \mapsto \alpha e^{-\rho s}
\left\{
\begin{array}{ll}
0 & \text{if }  \langle x, \psi \rangle \leq 0\\
\langle x, \psi \rangle^2  & \text{if } \langle x, \psi \rangle >0
\end{array}
\right .
\end{array}
\right .
\end{equation}
(that we could write in a more compact form as $v (s,x) = \alpha e^{-\rho s} \Theta \left ( \left \langle x, \psi \right\rangle \right ) \langle x, \psi \rangle^2$) is a strong solution of (\ref{eq:HJB-example}).
\end{Lemma}
\begin{proof}
We verify all the requirements of Definition \ref{def:strong-sol}. Given the form of $g$ in (\ref{eq:HJB-example}) one can easily see that $g\in C(H)$. The first derivatives of $v$ are given by
\[
\partial_s v(s,x) = -\rho \alpha e^{-\rho s} \Theta \left ( \left \langle x, \psi \right\rangle \right ) \langle x, \psi \rangle^2
\]
and
\begin{equation} \label{E41}
\partial_x v(s,x) = 2\alpha e^{-\rho s} \Theta \left ( \left \langle x, \psi \right\rangle \right ) \langle x, \psi \rangle\psi,
\end{equation}
so the regularities of $v$ demanded in the first two lines of Definition \ref{def:strong-sol} are easily verified. Injecting \eqref{E41} into 
\eqref{E40} yields
\begin{equation}
\label{eq:Findxv}
F(s,x,\partial_x v(s,x)) =
 - \alpha^2 \Theta \left ( \left \langle x, \psi \right\rangle \right ) \left \langle x, \psi \right\rangle^2 \left \langle \phi, \psi \right\rangle^2 e^{-\rho s},
\end{equation}
so the function  $(s,x) \mapsto F(s,x,\partial_xv(s,x))$
from $ \left[ 0,T\right] \times H$ to $H$ is finite and continuous.

We define, for any $n\in \mathbb{N}$, 
$\alpha_n := \frac{-\rho + (2+1/n) \lambda + \frac{1}{2}\beta^2 (2+1/n) (1+1/n)}{-\frac{1}{4} (2+1/n)^2 \langle \phi, \psi \rangle^2}$. We consider the  approximating sequence
\[
v_n(s,x) := \alpha_n e^{-\rho s} \Theta \left ( \left \langle x, \psi \right\rangle \right ) \langle x, \psi \rangle^{2+1/n}.
\]
The first derivative of $v_n$ w.r.t. $s$ and and first and second derivative of $v_n$ w.r.t. $x$ are given, respectively, by
\[
\partial_s v_n(s,x) = -\rho \alpha_n e^{-\rho s} \Theta \left ( \left \langle x, \psi \right\rangle \right ) \langle x, \psi \rangle^{2+1/n},
\]
\[
\partial_x v_n(s,x) = (2+1/n) \alpha_n e^{-\rho s} \Theta \left ( \left \langle x, \psi \right\rangle \right ) \langle x, \psi \rangle^{1+1/n} \psi
\]
and
\[
\partial_{xx}^2 v_n(s,x) = (2+1/n) (1+1/n) \alpha_n e^{-\rho s} \Theta \left ( \left \langle x, \psi \right\rangle \right ) \langle x, \psi \rangle^{1/n} \psi \otimes \psi.
\]
so it is straightforward to see that, for any 
$n \in \mathbb{N}$,  $v_n\in D(\mathscr{L}_0)$. Moreover, if we define
\begin{equation}
\label{eq:defgn} 
g_n(x):= e^{-\rho T} \alpha_n \Theta \left ( \left \langle x, \psi \right\rangle \right ) \left \langle x, \psi \right\rangle^{2+1/n}
\end{equation}
and 
\begin{equation}
\label{eq:defhn} 
h_n(s,x) := -\frac{1}{4} \alpha_n^2 e^{-\rho s} (2+1/n)^2 \Theta \left ( \left \langle x, \psi \right\rangle \right ) \left \langle \phi, \psi \right\rangle^2 \left \langle x, \psi \right\rangle^{2 + 1/n},
\end{equation}
(by an easy direct computation) we can see that $v_n$ is a classical solutions of the  problem
\[
\left\{
\begin{array}{l}
\mathscr{L}_0(v)(s,t) = h_n(s,x),\\ [8pt]
v(T,x)=g_n(x).
\end{array}
\right.
\]
The convergences asked in point (ii) of part (II) of Definition \ref{def:strong-sol} are straightforward.
\end{proof}

\begin{Lemma}
An optimal control of the problem (\ref{eq:state-example})-(\ref{eq:functional-example}) can be written in feedback form as
\begin{equation}
\label{eq:explicit-optimacontrol-example}
a(t) = -\alpha \Theta \left ( \left \langle \X(t), \psi \right\rangle \right ) \langle \X(t),\psi \rangle \langle \phi, \psi \rangle.
\end{equation}

The corresponding optimal trajectory is given by the unique solution of the mild equation
\begin{multline}
\label{eq:closedloop-example}
\X(t) = e^{(t-s)A} x - \int_s^t e^{(t-r)A} \phi \alpha \Theta \left ( \left \langle \X(r), \psi \right\rangle \right ) \langle \X(r),\psi \rangle \langle \phi, \psi \rangle dr \\
+ \beta \int_s^t e^{(t-r)A} \X(r) dW(r).
\end{multline}
\end{Lemma}
\begin{proof}
Observe that the hypotheses of Theorem \ref{th:verification} are verified: the regularity and the growth of $v$ are a simple consequence of its Definition (\ref{eq:defv-example}) and taking $b_i (t,x,a) = b(t,x,a) = a(t) \psi$ the condition (\ref{eq:conv-ucp-b}) is easily verified.



The optimality of (\ref{eq:explicit-optimacontrol-example}) is now just a consequence of point 2. of Remark \ref{RFeedback} once we observe that
\begin{eqnarray*}
 \quad \quad && \arg\min_{a\in \mathbb{R}} F_{CV}(s,x,\partial_xv(s,x);a) \\
&=&\arg\min_{a\in \mathbb{R}}\left\{  a 2\alpha e^{-\rho s} \Theta \left ( \left \langle x, \psi \right\rangle \right ) \langle x, \psi \rangle \left \langle \psi , \phi \right\rangle + e^{-\rho s} \Theta \left ( \left \langle x, \psi \right\rangle \right ) a^2 \right\}\\
&=&\left \{ 
\begin{array}{ll}
-\alpha \langle x,\psi \rangle \langle \phi, \psi \rangle & \text{if } \langle x,\psi \rangle \geq 0\\
\mathbb{R} & \text{if } \langle x,\psi \rangle < 0,
\end{array}
\right .
\end{eqnarray*}
so we can set 
$$\phi(s,x) = -\alpha \Theta \left ( \left \langle x, \psi \right\rangle \right ) \langle x,\psi \rangle \langle \phi, \psi \rangle.$$
\end{proof}
Observe that the elements $v_n$ of the approximating sequence are indeed the value functions of the optimal control problems having the same state equation (\ref{eq:state-example}) with running cost function 
$$ l_n(r,x,a) = e^{-\rho r} \Theta(\langle x, \psi\rangle) 
\langle x, \psi\rangle^{1/n} a^2
$$
and terminal cost function $g_n$ (defined in (\ref{eq:defgn})). The corresponding Hamiltonian is given by $(\blu{-}h_n)$ where $h_n$ is defined in (\ref{eq:defhn}).

\medskip

Even if it is rather simple example, it is itself of some interest because, as far as we know, no explicit (i.e. with explicit expressions of the value function and of the approximating sequence) example of \emph{strong solution} for second order HJB in infinite dimension is published so far.

\begin{Remark}
\label{rm:discussion-example}
In the example some non-regularities arise.
\begin{itemize}
 \item[(i)] The running cost function is 
$$l(r,x,a) = e^{-\rho r} \Theta(\langle x, \psi \rangle) a^2,$$ so
 for any choice of $a\neq 0$, it is discontinuous at any $x\in H$ such that $\left \langle x , \psi \right\rangle =0.$
 \item[(ii)] By \eqref{E40} the Hamiltonian $(s,x,p) \mapsto F(s,x,p)$ 
is not continuous
and even not finite. Indeed, for any non-zero $p\in H$ and for any $x\in H$ with $\left \langle x , \psi \right\rangle < 0$,
 its value is infinity. 
Conversely $F(s,x,\partial_xv(s,x))$ found in (\ref{eq:Findxv}) is always finite: observe that for any $x\in H$ with $\left \langle x , \psi \right\rangle \leq 0$, $\partial_xv(s,x)=0$.
 \item[(iii)] The second derivative of $v$ with respect to $x$ is well-defined on the points $(t,x)\in [0,T]\times H$ such that $\left \langle x , \psi \right\rangle <0$ (where its value is $0$) and it is well-defined on the points $(t,x)\in [0,T]\times H$ such that $\left \langle x , \psi \right\rangle>0$ (where its value is $2 \alpha e^{-\rho t} \psi\otimes\psi$) so it is discontinuous at all the points $(t,x)\in [0,T]\times H$ such that $\left \langle x , \psi \right\rangle =0$.
\end{itemize}
Thanks to points (i) and (ii) one cannot deal with the example by using existing results. Indeed, among various techniques, only solutions defined through a perturbation approach in space of square-integrable functions with respect to some invariant measure (see e.g. \cite{Ahmed03,GoldysGozzi06} and Chapter 5 of \cite{FabbriGozziSwiech16}) can deal with non-continuous running cost but they can (at least for the moment) only deal with problems with additive noise and satisfying the \emph{structural condition} and it is not the case here. Moreover none of the verification results we are aware of can deal at the moment with Hamiltonian with discontinuity in the variable $p$.
\end{Remark}

\color{black}

%
%


\bigskip

\bigskip

\bigskip
{\bf ACKNOWLEDGEMENTS:} The authors are grateful to the Editor and to the Referees
for having carefully red the paper and stimulated us to improve the first submitted version.
 This article was partially written during the stay of the
second named author 
at Bielefeld University, SFB 1283 (Mathematik).

\begin{small}

\bibliographystyle{plain}

\bibliography{bibliocontrol}

\end{small}

\end{document}